\documentclass[11pt,a4paper]{article}

\usepackage[english]{babel}
\usepackage[utf8x]{inputenc}
\usepackage[T1]{fontenc}
\usepackage[a4paper,top=3cm,bottom=2cm,left=3cm,right=3cm,marginparwidth=1.75cm]{geometry}
\usepackage{amsmath}
\usepackage{graphicx}
\usepackage{lineno,hyperref}
\usepackage{amsmath,amscd}
\usepackage{amsthm}
\usepackage{amsfonts}
\usepackage{color}
\usepackage{xfrac}
\usepackage{mathtools}
\usepackage{enumitem}
\usepackage[makeroom]{cancel}

\usepackage{mathrsfs}
\usepackage{tikz}
\usetikzlibrary{matrix,arrows,decorations.pathmorphing}
\tikzset{node distance=2cm, auto}

\newtheorem{theorem}{Theorem}[section]

\newtheorem{corollary}[theorem]{Corollary}
\newtheorem{lemma}[theorem]{Lemma}
\newtheorem{remark}[theorem]{Remark}

\newtheorem{definition}[theorem]{Definition}
\newtheorem{caveat}[theorem]{Caveat}
\newtheorem{summary}[theorem]{Summary}

\newcommand{\hookuparrow}{\mathrel{\rotatebox[origin=c]{90}{$\hookrightarrow$}}}

\title{The relationship of generalized manifolds to Poincar\'{e} duality complexes and topological manifolds}
\author{
  Friedrich Hegenbarth\textsuperscript{a} and Du\v{s}an Repov\v{s}\textsuperscript{b}}
\date{\today}

\begin{document}
\maketitle
\indent\textsuperscript{a} Dipartimento di Matematica "Federigo Enriques", 
Universit\`a degli Studi di Milano, IT-20133 Milano, Italy. \texttt{friedrich.hegenbarth@unimi.it}\\
\indent\textsuperscript{b} Faculty of Education and Faculty of Mathematics and Physics, University of Ljubljana, SI-1000 Ljubljana, Slovenia. \texttt{dusan.repovs@guest.arnes.si}

\begin{center}
\textit{Dedicated to the memory of Professor Sibe Marde\v{s}i\'{c} (1927-2016)}
\end{center}

\begin{abstract}
The primary purpose of this paper concerns the relation of (compact) generalized manifolds  to finite Poincar\'{e} duality complexes (PD complexes). 
The problem is that an arbitrary generalized manifold $X$ is always an ENR space, but it is not necessarily a complex. Moreover, finite PD complexes require the Poincar\'{e} duality with coefficients in the group ring $\Lambda$  ($\Lambda$-complexes). Standard homology theory implies that  $X$  is a $\mathbb{Z}$-PD   
complex.
Therefore by Browder's theorem, $X$  has a Spivak normal fibration which in turn, determines a Thom class of the pair $(N,\partial N)$ of a mapping cylinder neighborhood of $X$ in some Euclidean space. Then    $X$  satisfies the  $\Lambda$-Poincar\'{e} duality if this class induces an  isomorphism with $\Lambda$-coefficients. Unfortunately, the proof of Browder's theorem gives only isomorphisms with $\mathbb{Z}$-coefficients.
It is also not very helpful that $X$ is homotopy equivalent  to a finite complex $K$, because $K$ is not automatically a $\Lambda$-PD
complex. 
Therefore it is convenient to introduce $\Lambda$-PD
structures. To prove their existence on $X$,  we use the construction of  $2$-patch spaces and some fundamental results of Bryant, Ferry, Mio, and Weinberger. Since the class of all $\Lambda$-PD
complexes does not contain all generalized manifolds, we appropriately enlarge this class and then describe (i.e. recognize) generalized manifolds within this enlarged class in terms of the Gromov-Hausdorff metric.
\end{abstract}

\section{Introduction}\label{S:Introduction}
This paper deals with compact oriented  generalized manifolds, mostly without boundary, and of dimension $\geq 5$. Topological manifolds belong to this class. Conversely, by the well-known characterization theorem of Edwards~\cite{Edw78} and the resolution theorem of Quinn~\cite{Qui83,Qui87}, topological manifolds can be recognized in the class of generalized manifolds. This is briefly described in Section~\ref{S:Manifolds}, thereby stabilizing basic notations used later on.

A generalized manifold $X^n$ has the fundamental class $[X] \in H_n (X, \mathbb{Z})$ and it satisfies the Poincar\'{e} duality (PD) with respect to $\mathbb{Z}$ coefficients. However, it is not a Poincar\'{e} duality complex in the sense of Wall~\cite{Wall70} (see also Appendix).
Therefore it is appropriate to introduce the concept of the \textit{simple} $\Lambda$-PD
\textit{structure} on $X$, where $\Lambda = \mathbb{Z}[\pi_1 (X)]$ denotes the integral group ring of $\pi_1 (X)$. It consists of a (simple) symmetric algebraic Poincar\'{e} chain complex $(D_\#, \Phi)$ together with a (simple) chain equivalence
$$\alpha: (D_\# , \Phi)~\rightarrow~(S_\# (\widetilde{X}), \Delta [X]),$$ 
where $S_\# (\widetilde{X})$  denotes the singular chain complex of the universal cover $\widetilde{X}$ of $X$. We introduce this in Section~\ref{SS:Lambda}.

To construct such a structure the following is convenient: Let $N$ be a mapping cylinder neighborhood of an embedding $X \subset \mathbb{R}^m$, and $r: \partial N \rightarrow X$ the retraction.
Then $X$ is a $\mathbb{Z}$-PD space, but by Browder's theorem~\cite[Theorem~A]{Browd72}, the map $r: \partial N \rightarrow X$ has the structure of a spherical fibration, hence there is a Thom class, represented by a cycle 
$$[\mathcal{U}] \in C^{m-n} (N, \partial N, \mathbb{Z}).$$
Here, $C^{\#} (\cdot, G)$ stands for the cellular chain complex with coefficients in $G$. Let 
$$[\Sigma] = [\mathcal{U}] \cap [N] \in C_n (N, \mathbb{Z}),$$ 
where $[N] \in C_m (N, \partial N, \mathbb{Z})$ is the fundamental cycle of the manifold $N$. It is not obvious that $C_{\#} (N, \Lambda) = C_{\#} (\widetilde{N})$, together with $[\Sigma]$, determines a symmetric Poincar\'{e} chain complex (see Remark~\ref{Rem:SS3.2} in Section~\ref{SS:Lambda}). In Section~\ref{SS:Simple} we present a different approach, based on constructions used in~\cite{BFMW96,BFMW07}.

Roughly speaking, one approximates $X$ by Poincar\'{e} duality complexes $K$, obtained by gluing manifolds $W^n$ and $V^n$ along a controlled homotopy equivalence $y: \partial W \rightarrow \partial V$ between their boundaries. Applying Chapman's extensions of the Whitehead torsion theory to ANR spaces~\cite{Chapman77}, one obtains symmetric simple $\Lambda$-PD
chain complexes ($D_\# , \Phi$) with simple chain equivalences 
$$\alpha : (D_\# , \Phi) \rightarrow  (S_\# (\widetilde{X}), \Delta [X]).$$
In Section~\ref{SS:SimpleP}
these are called \textit{simple $\Lambda$-Poincar\'{e} duality types} on $X.$

Section~\ref{S:Recognizing} is an attempt to distinguish generalized manifolds. This requires a class of spaces containing all generalized manifolds, and then describe generalized manifolds in this class.
In Section~\ref{SS:Class} we introduce an appropriate class $\cal{B}$ consisting of compact separable metric ENR's $B$ satisfying the $\mathbb{Z}$-Poincar\'{e}
duality. Moreover, these spaces $B$ come equipped by simple $\Lambda$-PD
types. Of course, this class contains the class of finite $\Lambda$-PD  
complexes in the sense of Wall.

The main theorem of~\cite[Theorem~9.1]{BFMW07} leads to the following important characterization of generalized manifolds:\\

\noindent\textbf{Characterization Theorem.} \textit{If $X\in \cal{B}$
has formal dimension~$\geq 6$
then $X$ is a generalized $n$-manifold if the following is satisfied:
\begin{enumerate}
	\item[(1)] either $X$ is the limit of an inverse sequence
	 $$Y=\varprojlim \{K_{\varepsilon_1} \leftarrow K_{\varepsilon_2} \leftarrow \ldots\leftarrow K_{\varepsilon_i} \leftarrow \ldots\}$$ of $\varepsilon_i$-controlled $\Lambda$-PD
	 complexes of dimension $n$ and controlled homotopy equivalences $K_{\varepsilon_{i+1}} \rightarrow K_{\varepsilon_{i}}$  
	\item[(2)] or $X$ is the cell-like-image $f:Y\to X$ of a generalized manifold $Y$ of type (1).
\end{enumerate}}

In Section~\ref{SS:Controlled} we recall some definitions and facts about controlled PD spaces, controlled homotopy lifting properties and approximate fibrations. In summary, this leads to the formulation that ``generalized manifolds are controlled $\Lambda$-Poincar\'{e} complexes'', which however is not appropriate.

It follows from the Daverman and Husch theorem~\cite{DavHus84} that $X$ is a generalized manifold if $r: \partial N \rightarrow X$ is an approximate fibration, where
 $\partial N$ is the boundary of a mapping cylinder neighborhood $N$ of $X \subset \mathbb{R}^m$ (see also~\cite[Example~2.3]{Qui83}). Hence a necessary and sufficient condition for $X$ to be a generalized manifold is that the spherical Spivak
 fibration $\nu_X$ of $X$ reduces to an approximate fibration in a controlled manner, i.e. there is a controlled homotopy equivalence $\partial E \nu_X \rightarrow \partial N$ over $X$. This leads to the following:\\

\noindent\textbf{Recognition Criterion:} \textit{Suppose that 
$B$ belongs to the class $\mathcal{B}$. Then $B$ is a generalized manifold if $B$ admits an $\varepsilon$-$\Lambda$-PD
structure for all $\varepsilon > 0$.}\\

In Section~\ref{SS:4.3} we characterize generalized manifolds as isolated limits in the metric Gromov-Hausdorff space. If $B$ belongs to $\mathcal{B}$ then its isometry class is an element of
 this space. Given two elements $B, B'$ of $\mathcal{B}$, 
 the Gromov-Hausdorff distance $d_G(B, B') \in \mathbb{R}_+$ is well-defined.

The approximation of generalized manifolds by 2-patch spaces (see Lemma~\ref{L:3.8} below) leads to the following criterion:\\
 
\noindent\textbf{Gromov-Hausdorff Limit Criterion:} \textit{Any neighborhood of a generalized manifold $B$ contains non-generalized manifolds. Moreover, any generalized manifold $B$ is the limit of 2-patch spaces.}\\

Note that both limit criteria, the inverse limit criterion
and the Gromov-Hausdorff limit criterion, are consequences of the constructions in~\cite{BFMW07}, specifically Lemma 7.2 therein.

In a slightly modified spirit of Marde\v{s}i\'{c} and Segal \cite{MS1, MS2} one can state the following: An ANR
space $X$ is a generalized $n$-manifold if $X$ is a $\mathcal{C}_n$-space, where $\mathcal{C}_n$ is the class of $n$-dimensional 2-patch spaces. A more precise description can be given in terms of the Gromov-Hausdorff space which contains $\mathcal{C}_n$. It is a complete metric space with respect to the Gromov-Hausdorff metric $d_G$. 

We prove in Section~4.3
 that a compact generalized $n$-manifold is a limit of elements of $\mathcal{C}_n$ with respect to the metric $d_G$, i.e. the frontier of $\mathcal{C}_n$ in the Gromov-Hausdorff space consists of compact generalized $n$-manifolds 
(here $n\geq 6$, as usual).

\section{Manifolds and generalized manifolds}\label{S:Manifolds}

\subsection{Preliminaries}\label{SS:Definitions}
\begin{definition}
Let $X$ be a nonempty separable metric space and $k\ge 0$ any integer. Then 
\begin{itemize}
\item $X$ is said to have dimension at most $k$,
 ${\hbox{\rm dim}} X \leq k$, if for any open covering $\{\mathcal{U}_\alpha\}_{\alpha \in \mathcal{J}}$ of $X$ there is an open refinement $\{\mathcal{V}_i\}_{i \in \mathcal{L}}$ such that any $k+1$ elements of $\{\mathcal{V}_i\}_{i \in \mathcal{L}}$ have empty intersection;
 \item $X$ is said to be  $k$-dimensional, ${\hbox{\rm dim}} X = k$, 
 if ${\hbox{\rm dim}} X \leq k$ and ${\hbox{\rm dim}} X \not\leq k-1$.
 \item $X$ is said to be infinite-dimensional, ${\hbox{\rm dim}} X = \infty$, if ${\hbox{\rm dim}} X \not\leq k$ for every $k \geq 0$.
 \end{itemize}
\end{definition}

\begin{definition}
A topological space $X$ is called {a Euclidean neighborhood retract} (ENR) if $X$ embeds in some $m$-dimensional Euclidean space $\mathbb{R}^m$ as a closed subset so that there is a neighborhood $N \subset \mathbb{R}^m$ of $X$ which retracts onto $X$.
\end{definition}

 It's well-known that a separable metric space $X$ is an ENR if
$X$ is locally contractible and
	${\hbox{\rm dim}}X < \infty$.

\begin{definition}[see~\cite{CavHegRep16}]
A topological space
$X$ is called a {generalized $n$-manifold}, $n\in \mathbb{N}$,
if it satisfies the following properties:
\begin{itemize}
	\item[(i)] $X$ is an $n$-dimensional separable metric ENR; and
	\item[(ii)] for every $x\in X, H_\ast (X, X\setminus \{x\}, \mathbb{Z}) \cong H_\ast (\mathbb{R}^n, \mathbb{R}^n\setminus \{0\}, \mathbb{Z})$.
\end{itemize}

Furthermore,
$X$ is called a {generalized $n$-manifold with boundary} $\partial X \subset X$, 
if  $\partial X$ is also an ENR, and  the boundary $\partial X$ of $X$ is characterized by the following property $H_\ast (X, X\setminus \{x\}, \mathbb{Z}) \cong 0$ for every $x \in \partial X$ (see~\cite{Mit90}).
\end{definition}
Instead of ENR generalized manifolds, one often considers ANR generalized manifolds.
\begin{remark}[see~\cite{CavHegRep16}]

 For a separable metric space $X$ with ${\hbox{\rm dim}} X < \infty$, the following conditions are equivalent:
\begin{enumerate}
	\item[(1)] $X$ is an ANR.
	\item[(2)] $X$ is locally contractible.
	\item[(3)] $\hbox{\rm dim} X = k$ and $X$ is locally $k$-connected.
\end{enumerate}
\end{remark}

Remarkable properties of compact ENR's, hence of compact generalized manifolds, are expressed by the following result.
\begin{theorem}
The following properties are equivalent:
\begin{enumerate}[label=(\alph*)]
	\item $X$ is an ENR.
	\item For some $m$, there exist an embedding $\varphi: X \rightarrow \mathbb{R}^m$ and a mapping cylinder neighborhood $N \subset \mathbb{R}^m$ of $\varphi (X)$ in $\subset \mathbb{R}^m$.
	\item For some $n$, $X$ is the cell-like image of a compact manifold $M^n$ (possibly with boundary $\partial M^n$).
\end{enumerate}
\end{theorem}
\noindent
Here, (b) means that $N$ is homeomorphic to the mapping cylinder of a map $r: \partial N \rightarrow X$, denoted as
$\partial N \times I \underset{r}{\cup} X$. The homeomorphism $N \cong \partial N \times I \cup X$ is the identity on $\partial N$ and $X$. 
To explain (c), we recall the notion of a cell-like map
and some of its properties (see for instance~\cite{Fer92,Lac77}).

\begin{definition}
A compact subset $C \subset \hbox{\rm Int} M^n$, where $M^n$ is a topological $n$-manifold, is said to be \textit{cellular} in $M^n$  if it can be represented as follows:
$$C = \displaystyle\bigcap_{i = 1}^\infty B_i, \ \
\hbox{\rm where}\ \
B_i \subset \hbox{\rm Int} M^n\ \
\hbox{\rm  are}\ \
n\hbox{\rm -balls such that}\ \
B_{i+1} \subset \hbox{\rm Int} {B_i}, \ \
\hbox{\rm for}\ \
 i = 1, 2, \ldots.$$ 
 \end{definition}
 \begin{definition}
A surjective map 
$f: Y \rightarrow X$ is said to be \textit{cell-like} if for every $x \in X$, 
the preimage $f^{-1}(x)$ is a cell-like set,
i.e. there exists an embedding 
$$\varphi:f^{-1}(x)\hookrightarrow \hbox{\rm Int} M^n$$ 
into some topological $n$-manifold $M^n$ such that $\varphi(f^{-1}(x))$ is  cellular in $M^n$.
\end{definition}
The following characterization is very useful (see~\cite{Lac77}).

\begin{theorem}
Let $X$ and $Y$ be ENR spaces and $f: Y \rightarrow X$ a proper map. The following properties are equivalent:
\begin{enumerate}[label=(\alph*)]
	\item The map $f$ is  cell-like.
	\item For all open contractible subsets $K \subset X$, $f^{-1}(K) \subset Y$ is contractible.
	\item For all open subsets $U \subset X$, the restriction $f\hspace{-0.3em}\mid_{f^{-1}(U)}: f^{-1}(U) \rightarrow U$ is a proper homotopy equivalence.
\end{enumerate}
\end{theorem}

The properties of cell-like maps are related to controlled topology, and this is extremely important, because it links the resolution problem~\cite{Edw78} to Quinn's invariant~\cite{Qui87}. There is extensive literature on the subject, let us mention~\cite{BrFeMi96, CavHegRep16, Fer92, Fer94, Lac77, Qui79, Rep94}. Here are some definitions and properties.

\begin{definition}\label{mapping}
A mapping $f: Y \rightarrow X$ between compact ENR's $Y$ and $X$
is said to have:
\begin{itemize}

\item the $UV^k (\varepsilon)$-\textit{property}, 
where $k\in\mathbb{N}$ and $\varepsilon>0$,
 if for every pair $(K, L)$ of complexes $L\subset K$
of dimension $\leq k+1$ and every pair of maps $(\alpha, \alpha_0) : (K, L) \rightarrow (X, Y)$, the commutative diagram


\begin{center}
\begin{tikzpicture}
  \node (L) {$L$};
  \node (K) [below of=L] {$K$};
  \node (Y) [right of=L] {$Y$};
  \node (X) [right of=K] {$X$};
  \draw[->, font=\small] (L) to node {$\alpha_0$} (Y);
  \node (U) [node distance=1cm, below of=L] {$\cap$};
  \draw[->, font=\small] (K) to node {$\alpha$} (X);
  \draw[->, font=\small] (Y) to node {$f$} (X);
  \draw[->, font=\small, dashed] (K) to node {$\overline{\alpha}$} (Y);
\end{tikzpicture}
\end{center}

\noindent 
can be completed by a map $\overline{\alpha}: K \rightarrow Y$ such that $\overline{\alpha} |_L = \alpha_0$ and there is a homotopy $H: K\times I \rightarrow X$ between $f\circ \overline{\alpha}$ and $\alpha$ with tracks $\{H (x, t)| t\in I\}$ of diameter $<\varepsilon$.

\item   the $UV^k$-property,
where $k\in\mathbb{N}$,
 if it has the  $UV^k (\varepsilon)$-property for all $\varepsilon>0$.

\item   the $UV^\infty (\varepsilon)$-property,
where $\varepsilon>0$,
 if it has the $UV^k (\varepsilon)$-property for all $k\in\mathbb{N}$.

\item   the
 $UV^\infty$-property
  if it has the $UV^k$-property for all $k\in\mathbb{N}$.
\end{itemize}
\end{definition}

\begin{definition}\label{homotopy}
The homotopy $H: K \times I \rightarrow X$  in Definition~\ref{mapping}
is called an $\varepsilon$-homotopy. It is now obvious what an $\varepsilon$-homotopy equivalence means.\end{definition}

The following was proved in~\cite{Lac77}: 
\begin{theorem}
Let $f: Y \rightarrow X$ be a surjective map between compact ENR's. Then the following properties are equivalent:
\begin{enumerate}
	\item[(i)] $f: Y \rightarrow X$ is a cell-like map.
	\item[(ii)] $f: Y \rightarrow Xf$ is an $\varepsilon$-homotopy equivalence for all $\varepsilon >0$.
	\item[(iii)] $f: Y \rightarrow X$ is a $UV^\infty$-map.
\end{enumerate}
\end{theorem}

\subsection{Recognizing topological manifolds among generalized ma\-ni\-folds}\label{SS:Recognizing}

Obviously, every topological $n$-manifold is a generalized $n$-manifold. An answer to the converse problem was given by Edwards~\cite{Edw78}:

\begin{definition}
A metric space $X$ is said to have the disjoint disks property (DDP) if for any $\varepsilon >0$
 and any maps $\alpha_1, \alpha_2: D^2 \rightarrow X$ of the 2-disk $D^2$ into $X$, there exist maps $\beta_1, \beta_2: D^2 \rightarrow X$ such that $\text{dist}(\alpha_i (x), \beta_i (x)) < \varepsilon$ for all $x \in D^2, i \in \{1, 2\}$, and $\beta_1 (D^2) \cap \beta_2 (D^2) = \emptyset$. 
\end{definition}

\begin{theorem}
Let $X^n$ be a generalized $n$-manifold, $n\geq 5$, and suppose that $X^n$ satisfies the DDP. Then $X^n$ is homeomorphic to a topological $n$-manifold if there exist a topological $n$-manifold $M^n$ and a cell-like-map $f: M^n \rightarrow X^n$. In fact, such $f$ can then be approximated by homeomorphisms.
\end{theorem}

At this point, Quinn invented controlled surgery theory to construct maps $f_\varepsilon: M^n_\varepsilon \rightarrow X$ which are $\varepsilon$-homotopy equivalences (over $X$) between $n$-manifolds $M^n_\varepsilon$, $n\geq 5$.
Choosing a sequence $\{\varepsilon_i\}$ with $\varepsilon_i \rightarrow 0$, one can construct a ``telescope''-manifold $N^{n+1}$ with an end-manifold $M^n$ (applying End Theorem~\cite{Qui83}), and a map being an $\varepsilon$-homotopy equivalence for all $\varepsilon > 0$. Hence $f: M \rightarrow X$ is a cell-like map.

The manifolds $M^n_\varepsilon$ and the maps $f_\varepsilon : M^n_\varepsilon \rightarrow X$ are constructed by controlled surgery~(\cite{Fer94, Qui83, Qui87}, revised in~\cite{CavHegRep03}, see also~\cite{CavHegRep16}), starting from a controlled surgery problem $(g,b): V^n \rightarrow X^n$, where $b$ is an appropriate bundle map.

As usually, there are obstructions to complete surgery in the middle dimension to obtain an $\varepsilon$-homotopy equivalence. It turns out that there is only an integer obstruction $i(X)$ if one starts with an appropriate problem $(g,b): V^n \rightarrow X$ (see~\cite{Qui83, Qui87}).
Actually, the proof given in~\cite{Qui83} (and later corrected in~\cite{Qui87}) is of local nature.
The existence of a canonical controlled surgery problem $(g,b): V^n \rightarrow X$ follows from results in~\cite{FerPed95}.

\section{Poincar\'{e} duality structures on generalized manifolds and simple type}\label{Poincare}

A generalized manifold $X^n$ is not \`{a} priori a Poincar\'{e} duality complex in the sense of Wall~\cite{Wall70}. First of all, $X$ is not a CW
complex. Even if we know that $X^n$ is homotopy equivalent to a finite CW
complex $K$~\cite{West77}, it is not at all clear that $K$ is a Poincar\'{e} duality complex with respect to local coefficients, i.e. that $K$ satisfies the Poincar\'{e} duality with coefficients in 
$\Lambda = \mathbb{Z}[\pi_1 (K)]$ (shortly, $\Lambda$-PD complex).

It appears that the notion of ``$\Lambda$-PD complex structure'' is appropriate in this context. 
Moreover, Wall's definition requires that the PD
isomorphism is a simple equivalence on the chain complex level, i.e. its Whitehead torsion vanishes. For details we refer to~\cite{Chapman77, Chapman80, Coh73, Fer92, Fer94, Mil66, RhMaKe67}. In the sequel, $X$ will denote a compact oriented generalized $n$-manifold without boundary (if necessary, $n\geq 5$).

\subsection{The Poincar\'{e} duality over $\mathbb{Z}$}\label{SS:Poincare}

Sheaf-theoretical methods imply that there is a fundamental class $[X] \in H_n (X, \mathbb{Z})$ such that 
$$\cap [X] : H^q (X, \mathbb{Z}) \rightarrow H_{n-q} (X, \mathbb{Z})$$ 
is an isomorphism. A representing cycle of $[X]$ (also denoted $[X]$) defines a chain equivalence $S^q (X) \rightarrow S_{n-q} (X)$ on the singular (co-)chain level.

At this point we know that there is an embedding $X \subset \mathbb{R}^m$, for $m$ sufficiently large, 
 a mapping cylinder neighborhood $N$ of $X$, and a map
 $r: \partial N \rightarrow X$, which is homotopy equivalent to a spherical fibration (Spivak fibration) $\partial E \nu_X \rightarrow X$ (see~\cite[Theorem~A]{Browd72}). 

\subsection{$\Lambda$-PD complex structures}\label{SS:Lambda}

The neighborhood $N$ of $X$ is a (smooth) manifold with boundary, hence a simplicial complex. 
Let $C_\# (N, \Lambda)$ be the cellular chain complex of the universal cover $\widetilde{N}$ of $N$, and 
$$C^\# (N, \Lambda) = \text{Hom}_{\Lambda} (C_\# (N, \Lambda), \Lambda)$$ 
the cellular cochain complex. 
Similarly, $C_\# (N, \partial N, \Lambda)$ and 
$$C^\# (N, \partial N, \Lambda), \Lambda = \mathbb{Z}[\pi_1(X)].$$

Note that
 $C^\# (N, \Lambda)$
 is equivalent to
 $ C^\#_c (\widetilde{N})$, the cochain complex of $\widetilde{N}$ with compact support~(\cite[pp.~358--360]{CarEil56}).

Let $[N] \in C_m (N, \partial N)$ represent the fundamental class, and let $[\mathcal{U}] \in C^{m-n} (N, \partial N)$ represent the Thom class in $H^{m-n} (N, \partial N, \mathbb{Z})$ coming from the spherical fibration $\partial E \nu_X \rightarrow X$. Then we have the following obvious diagram


\begin{align}\label{D:1}
\begin{tikzpicture}[node distance=3cm, auto]
  \node (Sc) {$S^\#_c (\widetilde{X})$};
  \node (Sn) [below of=L] {$S_{n-\#} (\widetilde{X})$};
  \node (C) [right of=Sc] {$C^\# (N, \Lambda)$};
  \node (Cm) [node distance=4cm, right of=C] {$C_{m - \#} (N, \partial N, \Lambda)$};
  \node (Cn) [below of=Cm] {$C_{n-\#} (N, \Lambda)$};
  \draw[->, font=\small] (Sc) to node {$r^\#$} (C);
  \draw[->, font=\small] (Sc) to node {$\bullet \cap [X]$} (Sn);
  \draw[->, font=\small] (C) to node {$\cap [N]$} (Cm);
  \draw[->, font=\small] (Cm) to node {$\cap [\mathcal{U}]$} (Cn);
  \draw[->, font=\small, above] (Cn) to node {$r_\#$} (Sn);
  \draw[->, font=\small] (C) to node {$\cap [\Sigma]$} (Cn);
\end{tikzpicture}
\end{align}

\noindent
(see Remark~\ref{Rem:SS3.2} at the end of Section~\ref{SS:Lambda}).
Here, $[\Sigma] \in C_n (N)$ is the image of $[N] \in C_m (N, \partial N)$, under the map $C_m (N, \partial N) \rightarrow C_n (N)$.

\begin{lemma}[see~\cite{Ran92}]
The pair ($C_\# (N, \Lambda), [\Sigma]$) is a symmetric (algebraic) complex of dimension $n$.
\end{lemma}

\begin{proof}
For details we refer to~\cite{Ran92}. The symmetric property of the chain equivalence 
$$\cap [\Sigma] : C^\# (N, \Lambda) \rightarrow C_{n - \#}(N, \Lambda)$$
is defined as the image of $[\Sigma]$ under the usual diagonal approximation 
$$\Delta : C_n (N) \rightarrow W \underset{\mathbb{Z} [\sfrac{\mathbb{Z}}{2}]}{\otimes} (C_\# (N, \Lambda) \underset{\Delta}{\otimes} C_\# (N, \Lambda)),$$
where $W$ denotes the free $\mathbb{Z} [\sfrac{\mathbb{Z}}{2}]$-module resolutions of $\mathbb{Z}$ (the generator of $\sfrac{\mathbb{Z}}{2}$ acts on $\mathbb{Z}$ by multiplication by $-1$).
\end{proof}

\begin{remark}
$C_\# (N, \Lambda)$ is $\Lambda$-free and finitely generated, hence 
$$W \underset{\mathbb{Z} [\sfrac{\mathbb{Z}}{2}]}{\otimes} (C_\# (N, \Lambda) \underset{\Lambda}{\otimes} C_\# (N, \Lambda)) \cong
 \text{Hom}_{\mathbb{Z} [\sfrac{\mathbb{Z}}{2}]} 
  (W, \text{Hom}_{\Lambda} (C^\# (N, \Lambda), C_\# (N, \Lambda)).$$ 
The image of $[N]$ therefore gives a sequence of $\Lambda$-chain-maps 
$$\Phi_s : C ^r (N, \Lambda) \rightarrow C_{n- r +s} (N, \Lambda),\quad s = 0, 1, 2, \ldots ,$$
such that $\Phi_0 = \bullet \cap [\Sigma]$.
\end{remark}

\begin{definition}
A {\it symmetric algebraic $\Lambda$-chain complex} is a couple $(D_\#, \Phi)$, where $D_\#$ is a free $\Lambda$-complex, and $\Phi = (\phi_s)$ is an element in 
$$\text{Hom}_{\mathbb{Z} [\sfrac{\mathbb{Z}}{2}]} (W, \text{Hom}_{\Lambda} (D^\#, D_\#)).$$ 
It is a symmetric (algebraic) $\Lambda$-Poincar\'{e} complex if $\phi_0$ is a $\Lambda$-chain equivalence. If $\phi_0$ is a simple $\Lambda$-chain equivalence, $(D_\#, \Phi)$ is called a {\it simple} symmetric $\Lambda$-Poincar\'{e} complex.
\end{definition}
\begin{definition}
Let $X$ be a generalized manifold as above. A (simple) $\Lambda$-PD structure on $X$ is a commutative diagram


\begin{center}
\begin{tikzpicture}[node distance=2.5cm, auto]
  \node (Sc) {$S^\#_c (\widetilde{X})$};
  \node (Sn) [below of=L] {$S_{n-\#} (\widetilde{X})$};
  \node (D) [right of=L] {$D^\#$};
  \node (Dn) [right of=Sn] {$D_{n-\#}$};
  \draw[->, font=\small] (Sc) to node {$\alpha^\#$} (D);
  \draw[->, font=\small] (Sc) to node {$\bullet \cap [X]$} (Sn);
  \draw[->, font=\small] (Dn) to node [swap] {$\alpha_\#$} (Sn);
  \draw[->, font=\small] (D) to node {$\Phi_0$} (Dn);
\end{tikzpicture}
\end{center} 

\noindent
with $(D_\#, \Phi)$ a (simple) symmetric $\Lambda$-Poincar\'{e} complex and $\alpha$ a chain equivalence.
\end{definition}

\begin{lemma}\label{L:2}
Under the conditions as above, the restriction map $r: N \rightarrow X$ and $(C_\# (N, \Lambda), [\Sigma])$ determine a symmetric $\Lambda$-Poincar\'{e} structure on $X$. It is unique up to an equivalence.
\end{lemma}

A proof of 
 Lemma~\ref{L:2} will be given in Section~\ref{SS:Class}.

\begin{definition}
Two (simple) $\Lambda$-PD structures $(D_\#, \Phi)$, $(D_\#', \Phi')$ on $X$ are said to be equivalent if there is a chain equivalence $\gamma: D_\# \rightarrow D_\#'$ such that 


\begin{center}
\begin{tikzpicture}[node distance=2.5cm, auto]
  \node (D) {$D^\#$};
  \node (Dn) [below of=L] {$D_{n-\#}$};
  \node (Sc) [right of=L] {$S^\#_c$};
  \node (Sn) [right of=Dn] {$S_{n-\#} (\widetilde{X})$};
  \node (D') [right of=Sc] {$D'^\#$};
  \node (Dn') [right of=Sn] {$D'_{n-\#} $};
  \draw[->, font=\small, above] (Sc) to node {$\alpha^\#$} (D);
  \draw[->, font=\small] (D) to node {$\Phi_0$} (Dn);
  \draw[->, font=\small] (Dn) to node {$\alpha_\#$} (Sn);
  \draw[->, font=\small] (Sc) to node {$\cap [X]$} (Sn);
  \draw[->, font=\small] (Dn') to node [swap] {$\alpha'_\#$} (Sn);
  \draw[->, font=\small] (Sc) to node {$\alpha'^\#$} (D');
  \draw[->, font=\small] (D') to node {$\Phi_0'$} (Dn');
  \draw[->, bend right] (D') to node [swap] {$\gamma^\#$} (D);
  \draw[->, bend right] (Dn) to node {$\gamma$} (Dn');
\end{tikzpicture}
\end{center} 

\noindent
commutes, and $\gamma$ respects $\Phi$ and $\Phi'$ (for details see~\cite{Ran92}).

\end{definition}

Uniqueness of $(C_\# (N, \Lambda), [\Sigma])$ is due to the stability of the homotopy equivalence 
$$(N, \partial N) \sim (E \nu_X, \partial E \nu_X),$$ 
where $E \nu_X$ is the mapping cylinder of $\partial E \nu_X \rightarrow X$.

The question if $(C_\# (N, \Lambda), [\Sigma])$ is a simple $\Lambda$-PD structure reduces to whether 
$$\bullet \cap [\mathcal{U}]: C_\# (N, \partial N, \Lambda) \rightarrow C_{\# - (m-n)} (N, \Lambda)$$ 
is a simple $\Lambda$-chain equivalence. This will follow from an alternative approach presented in the following subsections.

\begin{remark}[\hbox{\rm Comments on~\cite[Theorem~A]{Browd72}}]\label{Rem:SS3.2}

In Diagram~(\ref{D:1}), the maps $r_\#$ and $r^\#$ are chain equivalences, and we want that $\bullet \cap [\Sigma]$ is also a chain equivalence. However, this cannot be deduced from  \cite[Theorem~A]{Browd72}. The strategy of the proof is to embed $X$ into a simply-connected Poincar\'{e} space $W^{n+1}$ of dimension $n+1$ (more precisely, $X\times I$ is embeded in $W^{n+1}$), and then apply Spivak's result to $W$, to get the fibration $\nu_W$. Then restricting it to $X$, one gets the Spivak fibration over $X$. Hence the Thom class of $\nu_X$ is the restriction of the Thom class of $\nu_W$, so the cap product with it inherits only a $\mathbb{Z}$-chain equivalence, and not a $\Lambda$-equivalence.
\end{remark}

\subsection{Simple $\Lambda$-PD  structures on $X$}\label{SS:Simple}

The construction here uses ideas and results from~\cite{BFMW96, BFMW07}. More specifically, \cite[Lemma~7.2]{BFMW07} contains the following fact:

\begin{lemma}\label{L:3.8}
Given $X$ as above with ${\hbox{\rm dim}} X = n \geq 6$, and given $\varepsilon > 0$, there is a space $X_\varepsilon$ and an $\varepsilon$-homotopy equivalence $X_\varepsilon \rightarrow X$ over $X$, where $X_\varepsilon = W \underset{\mathcal{S}}{\cup} V$ is patched together from $n$-manifolds with boundary $\partial W$ and $\partial V$ along an $\varepsilon$-homotopy equivalence $\mathcal{S}: \partial W \rightarrow \partial V$ over $X$
(see\cite{BFMW07}).
\end{lemma}
It is of course, required that $\pi_1 (\partial W) \cong \pi_1 (W)$ and $\pi_1 (\partial V) \cong \pi_1 (V)$. We may assume that $W$, $V$ are simplicial complexes and $\mathcal{S}: \partial W \rightarrow \partial V$ is simplicial, hence $X_\varepsilon$ is a complex. It is proved that $X_\varepsilon$ is an $\varepsilon$-controlled PD
complex. We are only interested if $X_\varepsilon$ is a simple $\Lambda$-PD
complex, which follows from the standard Mayer-Vietors argument and a vanishing Whitehead torsion. For convenience we write down the relevant diagrams: 


\begin{center}
\begin{tikzpicture}[node distance=2.5cm, auto]
  \node (dots_ul) {$\dots$};
  \node (Hq-1) [right of=dots_ul] {$H^{q-1} (W \cap V, \Lambda) $};
  \node (Hq) [node distance=5cm, right of=Hq-1] {$H^q (X_\varepsilon, W \cap V, \Lambda)$};
  \node (HqX) [node distance=4cm, right of=Hq] {$ H^q (X_\varepsilon, \Lambda)$};
  \node (dots_ur) [right of=HqX] {$\dots$};
      \node (cong) [node distance=1cm, below of=Hq] { };
  \node (HqW) [node distance=1cm, below of=cong] {$H^q (W, \partial W, \Lambda) \oplus H^q (V, \partial V, \Lambda)$};
  \node (dots_dl) [node distance=4cm, below of=dots_ul] {$\dots$};
  \node (Hn-q) [node distance=4cm, below of=Hq-1] {$H_{n-q} (W \cap V, \Lambda)$};
  \node (Hn-qW) [node distance=2cm, below of=HqW] {$ H_{n-q} (W, \Lambda) \oplus H_{n-q} (V, \Lambda)$};
  \node (Hn-qX) [node distance=4cm, below of=HqX] {$H_{n-q} (X_\varepsilon, \Lambda)$};
  \node (dots_dr) [node distance=4cm, below of=dots_ur] {$\dots$};
  \draw[->] (dots_ul) to node {} (Hq-1);
  \draw[->] (Hq-1) to node {} (Hq);
       \draw[->] (Hq) to node {$\cong$} (HqW);
  \draw[->] (Hq) to node {} (HqX);
  \draw[->] (HqX) to node {} (dots_ur);
  \draw[->] (dots_dl) to node {} (Hn-q);
  \draw[->] (Hn-q) to node {} (Hn-qW);
  \draw[->] (Hn-qW) to node {} (Hn-qX);
  \draw[->] (Hn-qX) to node {} (dots_dr);
  \draw[->] (Hq-1) to node {$\cong$} (Hn-q);
        \draw[->] (HqX) to node {$\cong$} (Hn-qX);
  \draw[->] (HqW) to node {$\cong$} (Hn-qW);
\end{tikzpicture}
\end{center} 

Here, we consider $W, V \subset X_\varepsilon$. The left vertical isomorphism fits into the diagram of $\Lambda$-PD
isomorphisms


\begin{center}
\begin{tikzpicture}[node distance=2.5cm, auto]
  \node (Hq-1) {$H^{q-1} (\partial W, \Lambda)$};
  \node (Hn-q) [below of=Hq-1] {$H_{n-q} (\partial W, \Lambda)$};
  \node (Hq-1W) [node distance=4cm, right of=Hq-1] {$H^{q-1} (W \cap V, \Lambda)$};
  \node (Hq-1V) [node distance=4cm, right of=Hq-1W] {$H^{q-1} (\partial V, \Lambda)$};
  \node (Hn-qW) [below of=Hq-1W] {$H_{n-q} (W \cap V, \Lambda)$};
  \node (Hn-qV) [below of=Hq-1V] {$H_{n-q} (\partial V, \Lambda)$};
  \draw[->, font=\small] (Hq-1W) to node [swap] {$\cong$} (Hq-1);
  \draw[->, font=\small] (Hq-1W) to node {$\cong$} (Hq-1V);
  \draw[->, font=\small] (Hn-q) to node {$\cong$} (Hn-qW);
  \draw[->, font=\small] (Hn-qV) to node [swap] {$\cong$} (Hn-qW);
  \draw[->, font=\small] (Hq-1) to node {$\cong$} (Hn-q);
  \draw[->, font=\small] (Hq-1W) to node {} (Hn-qW);
  \draw[->, font=\small] (Hq-1V) to node {$\cong$} (Hn-qV);
  \draw[->, bend right] (Hq-1V) to node [swap] {$\mathcal{S}^*$} (Hq-1);
  \draw[->, bend right] (Hn-q) to node {$\mathcal{S}_*$} (Hn-qV);
\end{tikzpicture}
\end{center} 


\begin{center}
\begin{tikzpicture}[node distance=2.5cm, auto]
  \node (Hq-1) {$C^{\#} (\partial W, \Lambda)$};
  \node (Hn-q) [below of=Hq-1] {$C_{n-1-\#} (\partial W, \Lambda)$};
  \node (Hq-1W) [node distance=4cm, right of=Hq-1] {$C^{\#} (W \cap V, \Lambda)$};
  \node (Hq-1V) [node distance=4cm, right of=Hq-1W] {$C^{\#} (\partial V, \Lambda)$};
  \node (Hn-qW) [below of=Hq-1W] {$C_{n-1-\#} (W \cap V, \Lambda)$};
  \node (Hn-qV) [below of=Hq-1V] {$C_{n-1-\#} (\partial V, \Lambda)$};
  \draw[->, font=\small] (Hq-1W) to node {} (Hq-1);
  \draw[->, font=\small] (Hq-1W) to node [swap] {} (Hq-1V);
  \draw[->, font=\small] (Hn-q) to node [swap] {} (Hn-qW);
  \draw[->, font=\small] (Hn-qV) to node {} (Hn-qW);
  \draw[->, font=\small] (Hq-1) to node [swap] {} (Hn-q);
  \draw[->, font=\small] (Hq-1W) to node {} (Hn-qW);
  \draw[->, font=\small] (Hq-1V) to node {} (Hn-qV);
  \draw[->, bend right] (Hq-1V) to node [swap] {$\mathcal{S}^\#$} (Hq-1);
  \draw[->, bend right] (Hn-q) to node {$\mathcal{S}_\#$} (Hn-qV);
\end{tikzpicture}
\end{center} 

\noindent
defining two simple $\Lambda$-PD structures on $W \cap V$. 

Now, $\mathcal{S}$ is an $\varepsilon$-equivalence, hence it follows by~\cite[Theorem~1']{Chapman80} that $\mathcal{S}_\#$, and $\mathcal{S}^\#$ are simple chain equivalences for sufficiently small $\varepsilon > 0$. It then follows by~\cite[Proposition~2.7]{Wall70}, that $X_\varepsilon$ is a simple $\Lambda$-PD complex. The induced chain equivalence of $X_\varepsilon \rightarrow X$ determines a simple $\Lambda$-PD
structure on $X$:


\begin{center}
\begin{tikzpicture}[node distance=3cm, auto]
  \node (L) {$S^\#_c (\widetilde{X})$};
  \node (K) [below of=L] {$S_{n-\#} (\widetilde{X})$};
  \node (Y) [right of=L] {$C^\# (X_\varepsilon, \Lambda)$};
  \node (X) [right of=K] {$C_{n-\#} (X_\varepsilon, \Lambda)$};
  \draw[->, font=\small] (L) to node {} (Y);
  \draw[->, font=\small] (L) to node {$\bullet \cap [X]$} (K);
  \draw[->, font=\small] (X) to node [swap] {} (K);
  \draw[->, font=\small] (Y) to node {$\bullet \cap [X_\varepsilon]$} (X);
\end{tikzpicture}
\end{center}

\subsection{The simple $\Lambda$-Poincar\'{e} duality type}\label{SS:SimpleP}

The simple $\Lambda$-PD structure on a generalized manifold $X$, given by $X_\varepsilon \rightarrow X$, can be improved by requiring that it is a simple homotopy equivalence. However, this requires the extension of the Whitehead torsion theory to ENR spaces as done in~\cite{Chapman77}:
To any homotopy equivalence $f: X \rightarrow Y$ between ANR's one can assign an element $\overline{\tau} (f) \in Wh (\pi_1 (M))$ such that:

\begin{enumerate}
	\item[(i)] If $X$, $Y$ are finite CW complexes then $\overline{\tau} (f) = \tau (f)$, where $\tau$ denotes the classical torsion.
	\item[(ii)] $\overline{\tau} (f) = 0$ if and only if there exists an ANR $Z$ and cell-like maps $Z \xrightarrow{\alpha} X$, $Z \xrightarrow{\beta} Y$ such that 

\begin{center}
\begin{tikzpicture}[node distance=2.5cm, auto]
  \node (LU) {};
  \node (X) [below of=L] {$X$};
  \node (Z) [right of=L] {$Z$};
  \node (RU) [right of=Z] {};
  \node (Y) [below of=RU] {$Y$};
  \draw[->, font=\small] (Z) to node [swap] {$\alpha$} (X);
  \draw[->, font=\small] (X) to node {f} (Y);
  \draw[->, font=\small] (Z) to node {$\beta$} (Y);
\end{tikzpicture}
\end{center}

commutes up to homotopy. Any ANR $X$ has a simple homotopy type given by $Id: X \rightarrow X$.

\end{enumerate}

Applying this to our situations we obtain that $X_\varepsilon \rightarrow X$ is a simple homotopy equivalence because it is an $\varepsilon$-homotopy equivalence, assuming that $\varepsilon >0$ is sufficiently small~\cite[Theorem~1']{Chapman80}.

If $N$ is the mapping cylinder neighborhood of $r: \partial N \rightarrow X$, then $r: N \rightarrow X$ is a simple homotopy equivalence with inverse the inclusion $i: X \hookrightarrow N$. One has 
$$\overline{\tau} (X_\varepsilon \rightarrow X \xhookrightarrow{i} N) = \overline{\tau} (X_\varepsilon \rightarrow X) + \overline{\tau} (i) = 0.$$

\begin{corollary}
The map 
$$\bullet \cap [\Sigma]: C^\# (N, \Lambda) \rightarrow C_{n-\#} (N, \Lambda)$$ 
is a simple chain equivalence.
\end{corollary}

\begin{proof}
This follows from the diagram


\begin{center}
\begin{tikzpicture}[node distance=2.5cm, auto]
  \node (L) {$C^{\#} (N, \Lambda)$};
  \node (K) [below of=L] {$C^{\#} (X_\varepsilon, \Lambda)$};
  \node (Y) [node distance=4cm, right of=L] {$C_{n-\#} (N, \Lambda)$};
  \node (X) [node distance=4cm, right of=K] {$C_{n-\#} (X_\varepsilon)$};
  \draw[->, font=\small] (L) to node {$\bullet \cap [\Sigma]$} (Y);
  \draw[->, font=\small] (K) to node {} (L);
  \draw[->, font=\small] (K) to node {$\bullet \cap [X_\varepsilon]$} (X);
  \draw[->, font=\small] (Y) to node {} (X);
\end{tikzpicture}
\end{center}
\end{proof}

\begin{definition}
A {\it simple $\Lambda$-PD type} on $X$ is defined as a simple symmetric $\Lambda$-PD
structure $(D_\#, \Phi)$


\begin{center}
\begin{tikzpicture}[node distance=2.5cm, auto]
  \node (L) {$S^{\#}_c (\widetilde{X})$};
  \node (K) [below of=L] {$S_{n-\#} (\widetilde{X})$};
  \node (Y) [node distance=4cm, right of=L] {$D^\#$};
  \node (X) [node distance=4cm, right of=K] {$D_{n-\#}$};
  \draw[->, font=\small] (L) to node {$\alpha^\#$} (Y);
  \draw[->, font=\small] (L) to node {$\bullet \cap [X]$} (K);
  \draw[->, font=\small] (X) to node [swap] {$\alpha_\#$} (K);
  \draw[->, font=\small] (Y) to node {$\Phi_0$} (X);
\end{tikzpicture}
\end{center}

\noindent
with $\alpha$ a simple chain equivalence. Two types are said to be {\it equivalent} if the chain equivalence $ D_{\#} \xrightarrow[]{\gamma} D_{\#}'$ is simple. 
\end{definition}
In this sense we have the following:

\begin{summary}
Any compact oriented generalized $n$-manifold $X$ has a simple  $\Lambda$-PD
type, unique up to equivalence determined by a mapping cylinder neighborhood $N$ of $X \subset \mathbb{R}^m$, where $n \geq 6$ and $m$ is sufficiently large.
\end{summary}

\begin{remark}\label{R:3.12}
Strictly speaking, it does not make sense to state that $\alpha: D_\# \rightarrow S_\# (\widetilde{X})$ is a simple chain equivalence.
However,
 the extension of Whitehead torsion to compact ANR spaces is built on the fact that there is a cell-like map $K \times Q \rightarrow X$, where $Q$ is a Hilbert cube manifold and $K$ is a finite complex. Then ``simple'' refers to the finite complex $K$. 
\end{remark}

\section{Recognizing generalized manifolds among ENR spaces with simple $\Lambda$-PD
type}\label{S:Recognizing}

\subsection{Controlled Poincar\'{e} duality complexes and approximate fibrations}\label{SS:Controlled}

Controlled Poincar\'{e} duality complexes are a bridge between  simple $\Lambda$-PD complexes and generalized manifolds. There is a controlled geometric aspect linked to~\cite[Proposition~4.5]{BFMW96} (see also~\cite[Example~2.3]{Qui83}). The geometric aspect leads to approximate fibrations introduced in~\cite{CorDuv77}.

\begin{definition}
An \textit{oriented $n$-dimensional} $\varepsilon$-PD
\textit{complex} $K$ (over $K$) is a finite complex $K$, an $n$-cycle $[K] \in C_n (K)$ such that cap product with it 
$$\bullet \cap [K]: C^\# (K, \Lambda) \rightarrow C_{n-\#} (K, \Lambda)$$
is an $\varepsilon$-chain equivalence over $K$.
\end{definition}
To define $\varepsilon$-chain maps, resp. $\varepsilon$-chain equivalences, one needs the notion of geometric chain complexes. We refer to the literature, in particular to~\cite{Yam87} and~\cite[Remark~2 on p.120]{Yam98} and~\cite[Definition~2.1]{Fer94}.
One observes that the $\varepsilon$-Poincar\'{e}-duality is a much stronger condition than the simple Poincar\'{e} duality, assuming that $\varepsilon$ is sufficiently small (again~\cite[Theorem~1']{Chapman80}).

Let $N \subset \mathbb{R}^m$ be a regular neighborhood of $K \subset \mathbb{R}^m$, $m$ sufficiently large, $r: N \rightarrow K$ the restriction. Then by~\cite[Proposition~4.5]{BFMW96}, the restrictions of $r$ to $\partial N$ (also denoted by $r$), $r: \partial N \rightarrow K$ , has the $\delta$-controlled homotopy lifting property. Here $\varepsilon$ is sufficiently small
 and $\delta$ depends on $\varepsilon$ and $K$.

\begin{definition}
The \textit{$\delta$-controlled homotopy lifting property} is defined as follows:
Given a separable metric space $Z$, and a commutative diagram


\begin{center}
\begin{tikzpicture}[node distance=2.5cm, auto]
  \node (L) {$Z \times \{0\}$};
  \node (K) [below of=L] {$Z \times I$};
  \node (Y) [node distance=3cm, right of=L] {$\partial N$};
  \node (X) [node distance=3cm, right of=K] {$K$};
  \draw[->, font=\small] (L) to node {$h_0$} (Y);
  \draw[->, font=\small] (L) to node {$i$} (K);
  \draw[->, font=\small] (K) to node {$h_t$} (X);
  \draw[->, font=\small] (Y) to node {$r$} (X);
  \draw[->, font=\small] (K) to node {$H$} (Y);
\end{tikzpicture}
\end{center}
\noindent
there is a homotopy $H: Z \times I \rightarrow \partial N$, such that $\text{dist} (r \circ H_t, h_t) < \delta$.
(Note the relationship with the ${UV}^{k}(\varepsilon)$-property in Definition~\ref{mapping}. For details we refer to~\cite{BrFeMi96}, in particular Theorem 6 therein.)
\end{definition}

We emphasize
 that this holds for some $\delta = T \cdot \varepsilon$, where $T > 0$ is a factor depending on $K$~\cite[Proposition~4.5]{BFMW96}.
However, if this holds for all $\delta > 0$ then one obtains approximate fibrations, more precisely (see~\cite{CorDuv77}):
\begin{definition}
An \textit{approximate fibration} $p: E \rightarrow B$ is a surjective map between locally compact separable metric ANR's $E$ and $B$ such that the above $\delta$-homotopy lifting property holds for all separable metric spaces $Z$ and  all $\delta > 0$.\end{definition}
It is remarkable and useful to note that it is sufficient to require the lifting property only for cells~\cite[Theorem~2.2]{CorDuv77}. 
The last building stone of the bridge is the following:

\begin{theorem}[\cite{DavHus84}]\label{T:4.4}
Suppose that $p: M \rightarrow B$ is a proper map, where $M$ is a closed connected manifold and $B$ is an ANR. Then, if $p$ is an approximate fibration, $B$ is a generalized manifold.
\end{theorem}

Putting this together, one can phrase the recognition principle as follows: 
A $\Lambda$-PD
complex is a generalized manifold if it satisfies the $\varepsilon$-Poincar\'{e}
duality for all $\varepsilon > 0$.
However, generalized manifolds in general do not have CW or simplicial structures.
It is therefore appropriate to define a class of spaces to which the above theorems can be applied to characterize generalized manifolds in this class.
 This will be done in the next section.


\subsection{The class of simple $\Lambda$-PD
types and approximations by controlled PD
complexes}\label{SS:Class}
In this section we consider the following class $\mathcal{B}$ of $PD$ spaces $B$ characterized by
\begin{enumerate}
	\item[(i)] $B$ is a compact separable metric ENR which satisfies $\mathbb{Z}$-Poincar\'{e} duality of dimension $n$.
	\item[(ii)] $B$ has a simple $\Lambda$-PD  type, i.e. there is a diagram

\begin{center}
\begin{tikzpicture}[node distance=2.5cm, auto]
  \node (L) {$S^\#_c (\widetilde{B})$};
  \node (K) [below of=L] {$S_{n-\#} (\widetilde{B})$};
  \node (Y) [node distance=3.5cm, right of=L] {$D^\#$};
  \node (X) [node distance=3.5cm, right of=K] {$D_{n-\#}$};
  \draw[->, font=\small] (L) to node {$\alpha^\#$} (Y);
  \draw[->, font=\small] (L) to node {$\bullet \cap [B]$} (K);
  \draw[->, font=\small] (X) to node [swap] {$\alpha_\#$} (K);
  \draw[->, font=\small] (Y) to node {$\Phi_0$} (X);
\end{tikzpicture}
\end{center}
	where $(D_\#, \Phi)$ is a simple symmetric algebraic Poincar\'{e} duality chain complex and $\alpha$ is a simple chain equivalence $(\Lambda = \mathbb{Z}[\pi_1 B])$,
	see Remark~\ref{R:3.12}). We denote this class by $\mathcal{B}$.
\end{enumerate}

Let $N \cong \partial N \times I \underset{r}{\cup} B$ be a mapping cylinder neighborhood of $B \subset \mathbb{R}^m$. It is equivalent to the Spivak fibration by~\cite{Browd72}, and it gives rise to the following diagram (see Section~\ref{Poincare}) 


\begin{center}
\begin{tikzpicture}[node distance=2.5cm, auto]
  \node (Hq-1) {$C^\# (N, \Lambda)$};
  \node (Hn-q) [below of=Hq-1] {$C_{n-\#}  (N, \Lambda)$};
  \node (Hq-1W) [node distance=4cm, right of=Hq-1] {$S^\#_c (\widetilde{B})$};
  \node (Hq-1V) [node distance=4cm, right of=Hq-1W] {$D^\#$};
  \node (Hn-qW) [below of=Hq-1W] {$S_{n-\#} (\widetilde{B})$};
  \node (Hn-qV) [below of=Hq-1V] {$D_{n-\#}$};
  \draw[->, font=\small] (Hq-1W) to node [swap] {$r^\#$} (Hq-1);
  \draw[->, font=\small] (Hq-1W) to node {$\alpha^\#$} (Hq-1V);
  \draw[->, font=\small] (Hn-q) to node {$r_\#$} (Hn-qW);
  \draw[->, font=\small] (Hn-qV) to node [swap] {$\alpha_\#$} (Hn-qW);
  \draw[->, font=\small] (Hq-1) to node {$\bullet \cap [\Sigma]$} (Hn-q);
  \draw[->, font=\small] (Hq-1W) to node {$\bullet \cap [B]$} (Hn-qW);
  \draw[->, font=\small] (Hq-1V) to node {$\Phi_0$} (Hn-qV);
\end{tikzpicture}
\end{center} 

Since $r: N \rightarrow B$ is a simple equivalence, $\bullet \cap [\Sigma]$ is also a simple chain equivalence.
This proves Lemma~\ref{L:2} from Section~\ref{SS:Lambda}.

As explained in Section~\ref{Poincare}, generalized manifolds belong to this class, and of course, also finite simple $\Lambda$-PD
complexes.
To recognize generalized manifolds within the above defined class $\mathcal{B}$, we define $\varepsilon$-$\Lambda$-structures:
\begin{definition}
A (simple) $\Lambda$-PD structure $(D_\#, \Phi)$ on $B$ is a (simple) $\varepsilon$-$\Lambda$-PD
structure on $B$ if it makes $\bullet \cap [\Sigma]: C^\# (N) \rightarrow C_{n-\#} (N)$ a (simple) $\varepsilon$-chain equivalence by means of the diagram
\begin{center}
\begin{tikzpicture}[node distance=2.5cm, auto]
  \node (Hq-1) {$C^\# (N, \Lambda)$};
  \node (Hn-q) [below of=Hq-1] {$C_{n-\#}  (N, \Lambda)$};
  \node (Hq-1W) [node distance=4cm, right of=Hq-1] {$S^\#_c (\widetilde{B})$};
  \node (Hq-1V) [node distance=4cm, right of=Hq-1W] {$D^\#$};
  \node (Hn-qW) [below of=Hq-1W] {$S_{n-\#} (\widetilde{B})$};
  \node (Hn-qV) [below of=Hq-1V] {$D_{n-\#}$};
  \draw[->, font=\small] (Hq-1W) to node [swap] {$r^\#$} (Hq-1);
  \draw[->, font=\small] (Hq-1W) to node {$\alpha^\#$} (Hq-1V);
  \draw[->, font=\small] (Hn-q) to node {$r_\#$} (Hn-qW);
  \draw[->, font=\small] (Hn-qV) to node [swap] {$\alpha_\#$} (Hn-qW);
  \draw[->, font=\small] (Hq-1) to node {$\bullet \cap [\Sigma]$} (Hn-q);
  \draw[->, font=\small] (Hq-1W) to node {$\bullet \cap [B]$} (Hn-qW);
  \draw[->, font=\small] (Hq-1V) to node {$\Phi_0$} (Hn-qV);
\end{tikzpicture}
\end{center} 
\end{definition}

A result of Daverman and Husch (see Theorem~\ref{T:4.4} above)
leads to the following:

\begin{theorem}[Recognition Criterion]\label{RC}
Suppose that $B$ belongs to $\mathcal{B}$. 
Then $B$ is a generalized manifold if for every $\varepsilon > 0$, 
$B$ admits an $\varepsilon$-$\Lambda$-PD structure.
\end{theorem}

\begin{remark}
Clearly, $B$ 
in Theorem~\ref{RC} is  simple, if the $\varepsilon$-$\Lambda$-PD structures are simple.
\end{remark}
\begin{proof}
The hypothesis implies that $B$ is an $\varepsilon$-PD space for all $\varepsilon > 0$. By~\cite[Proposition~4.5]{BFMW96}, it follows that $\partial N \rightarrow B$ has the $T \cdot \varepsilon$-lifting property for all $\varepsilon > 0$, where $T$ 
is a factor depending on $B$. Hence $\partial N \rightarrow B$ is an approximate fibration, i.e. $B$ is a generalized manifold (see Theorem~\ref{T:4.4} above).
\end{proof}

\subsection{Generalized manifolds as limits in the Gromov-Hausdorff metric space}\label{SS:4.3}

Bryant-Ferry-Mio-Weinberger lemma \cite[Lemma~7.2]{BFMW07} also allows a
characterization of generalized manifolds in terms of the Gromov-Hausdorff metric, denoted here by $d_G$. It is defined on the set of isometry classes of compact metric spaces, and
 $d_G$ can be
  proved to be a complete metric (for details
see~\cite{Fer92}, or \cite[Section~1.V]{Fer94}).

We consider elements of the class $\mathcal{B}$ as being points
 of the Gromov-Hausdorff space and we prove the following result.

\begin{lemma}\label{patch}
Let $X$ be a generalized compact $n$-manifold and let $\delta > 0$ be given. Then there exists a 2-patch space $X'$ such that $d_G (X, X') < \delta$.
\end{lemma}
\begin{remark}
Note that $X'$ in Lemma~\ref{patch} is in general  not a generalized manifold.
\end{remark}

\begin{proof}
Let $N$ be a cylindrical (regular) neighborhood of an embedding $X \subset \mathbb{R}^m$, $m \geq 2n+1$, $m-n\geq 3$.
If $C \subset \mathbb{R}^m$ is compact, we denote by $\varrho(x, C)$ the distance of $x$ from $C$ in the metric of $\mathbb{R}^m$, and $C_\varepsilon = \{x \in \mathbb{R}^m | \varrho (x, C) < \varepsilon\}$.
\\
By Lemma~\ref{L:3.8} we can choose an $\varepsilon$-homotopy equivalence $f: X' \rightarrow X$.
We can choose $\varepsilon>0$ small enough to get $X_\varepsilon \subset N$. Inside $X_\varepsilon$ we can find a cylindrical (regular) neighborhood $Z \xrightarrow{p} X$. By Proposition~\cite[Proposition~4.10]{BFMW96}, there is an embedding $X' \xrightarrow{j} Z$ and a retraction $r: Z \rightarrow X'$ such that $\varrho(p(z), r(z)) < 2\varepsilon$ for $z \in Z$.
\\
For clarity we write down the maps in the diagram

\begin{center}
\begin{tikzpicture}[node distance=1.5cm, auto]
  \node (Z) {$Z$};
  \node (E) [node distance=1cm, right of=Z] {$=$};
  \node (ZZ) [node distance=1cm,right of=E] {$Z$};
  \node (C) [node distance=1cm,right of=ZZ] {$\subset$};
  \node (XE) [node distance=1cm,right of=C] {$X_\varepsilon$};
  \node (X') [below of=Z] {$X'$};
  \node (X) [below of=ZZ] {$X$};
  \draw[->, font=\small] (Z) to node [swap] {$r$} (X');
  \draw[->, font=\small] (Z) to node {$\hookuparrow j$} (X');
  \draw[->, font=\small] (ZZ) to node [swap] {$p$} (X);
  \draw[->, font=\small] (ZZ) to node {$\hookuparrow i$} (X);
  \draw[->, font=\small] (X') to node {$f$} (X);
\end{tikzpicture}
\end{center}

\noindent
For $z\in Z$ we therefore get
$$\varrho(z, r(z)) \leq \varrho(z, p(z)) +\varrho(p(z), r(z)) < \varrho(z, p(z)) + 2\varepsilon < 3\varepsilon,$$
since $Z \subset X_\varepsilon$.
Hence $\varrho(z, X') < 3\varepsilon,$ for all $z\in Z$, i.e. $X \subset Z \subset X'_{3\varepsilon}$. 
Since $X'\subset Z\subset X_\varepsilon$
we get by definition of $d_G$ that
$d_G(X',X)<3\varepsilon.$
This proves the lemma, since we can choose $\varepsilon$ arbitrarily small.
\end{proof}

\begin{caveat}
The notion $X_\varepsilon$ used in the above proof differs from the notion used in Section~\ref{SS:Simple}.
\end{caveat}
\begin{corollary}
Generalized manifolds are isolated limits of the subspace of 2-patch spaces in the Gromov-Hausdorff space.
\end{corollary}

\begin{remark}
Other examples of limits in the Gromov-Hausdorff space were considered in~\cite[Section 1.V]{Fer94}.
\end{remark}

\section*{Appendix}
The relations between $\mathbb{Z}$-PD and $\Lambda$-PD of a space $K$ were always behind the discussion in the previous sections.
Since it is of general interest and is not explicitly presented in the literature, we state them in the following lemma:

\begin{lemma}
Let $K$ be a finitely dominated CW
complex. If $K$ satisfies $\mathbb{Z}$-PD, then it satisfies $\Lambda$-PD, where $\Lambda = \mathbb{Z}[\pi_1(K)]$.\end{lemma}

For the proof we note the following key facts:
\begin{itemize}
	\item By Browder's Theorem~\cite[Theorem~A]{Browd72}, $K$ has a Spivak normal fibration $\pi: E\nu_K \rightarrow K$.
	\item  A result of Ranicki \cite[Proposition~3.10]{1980}
	 shows that the cap
	 product with the Thom class $\left[\mathcal{U}\right]$, 
	 $$\cap: H_\ast (T\nu_K, \Lambda) \rightarrow H_{\ast-k} (K, \Lambda)$$ 
	 is an isomorphism, where $k$ is the fiber dimension of $E\nu_K$.
	The proof of this consists of applying Browder's lemma \cite[Lemma~I.4.3]{1972}   to the universal cover $\widetilde{K}$ of $K$.
	\item Using the homotopy equivalence $(N, \partial N) \rightarrow (E\nu_K, \partial E\nu_K)$, where $N \subset \mathbb{R}^{n+k}$ is a regular neighborhood of $K \subset \mathbb{R}^{n+k}$, one obtains  $\Lambda$-PD as the following composite map
	$$H^\ast (K, \Lambda) \cong H^\ast (N, \Lambda) \rightarrow H_{n+k-\ast} (N, \partial N, \Lambda) \xrightarrow{\cap\left[\mathcal{U}\right]}H_{n-\ast} (K, \Lambda)\qed$$
\end{itemize}

We emphasise that this lemma is not helpful for proving the results of this paper.
 
\section*{Acknowledgements}
This research was supported by the Slovenian Research Agency grants P1-0292, J1-8131, J1-7025, J1-6721, and N1-0064. 
We thank K. Zupanc for her technical assistance with the preparation of the manuscript. We also acknowledge several comments from the referee.

\end{document}